\title{Series and Product Relations Made from Primes}
\author{Ken Hicks
        \&
        Kevin Ward}
\documentclass{article}

\usepackage{graphicx}

\newtheorem{theorem}{Theorem}
\newtheorem{lemma}{Lemma}

\newenvironment{proof}{{\sc Proof:}}{~\hfill QED}

\begin{document}
\newpage
\maketitle
%%%%%%%%%%%%%%%%%%%%%%%%%%%
% abstract
%%%%%%%%%%%%%%%%%%%%%%%%%%%
\begin{abstract}
It is known that the sum of the reciprocal of integers, $\sum_n (1/n)$, 
and the sum of the reciprocal of primes, $\sum_n (1/p_n)$, both diverge. 
Here, we study a series made from primes that sums exactly to 1. 
We also show this sum is simply related to an infinite product over primes. 
We then generalize the form of the series (and its related product), 
extending this idea to other number theoretic sets such as twin primes.
Evaluating a product made from twin primes gives a 
new mathematical constant.
\end{abstract}

\section{Introduction}
Prime numbers hold a unique place in number theory.  These are the basic units of multiplication, yet there are many things we still don't know about prime numbers.  For example, it is fairly easy to show that there are an infinite number of primes, yet we still don't know if there are an infinite number of twin primes (i.e., prime pairs with a difference of 2).  Another interesting fact is that the sum over reciprocals of primes diverges, yet it has been proven that the sum over reciprocals of twin primes converges (to a mathematical constant called Brun's constant).  There are still many things to be learned about the distributions of prime numbers and of twin primes, and also new mathematical constants to be discovered.

While investigating the distribution of primes and twin primes, we stumbled upon a sum of rational fractions, made only from prime numbers, that converged exactly to unity. Surprisingly, this series did not show up in a literature search\footnote{We later found this series at the OEIS, see sequence A005867, but without any proof \cite{OEIS}.}.  There are numerous examples of integer series in standard textbooks of number theory, but few series that involve primes (or twin primes). We presented this series to a number of professional mathematicians, some with expertise in number theory, but none had seen it before.  Here, we present a general proof of the convergence of this series, and apply this theorem to several related series.  When applied to twin primes, a new mathematical constant emerges.

The motivation for this work comes from a study of the sieve of Erostosthenes.  The use of a sieve to enumerate prime numbers goes back to the time of Euclid, many thousands of years ago.  In its simplest form, the integers greater than 1 are written in a long line, then factors of two are removed (marking out every other number, leaving just the odd numbers), then factors of three are removed (marking out every third number starting with 3), and so on.  It is easy to see that the first step removes half of all integers, the second step removes on-third of those remaining, and so forth.  The series given here is based on counting the fraction of all numbers removed by this sieve process.  

In some sieve methods, it is natural to define a primorial number, often denoted as $p\#$, which is like a factorial but multiples only primes instead of integers.  For example, the primorial of the first three primes is $p_3\# = 2 \cdot 3 \cdot 5= 30$.  The set of all integers up to a given primorial $p_n \#$ is useful as a finite set, which is often easier to work with than for the infinite set of integers.  For example, when sieving the numbers up to $p_2 \#=6$, sieving by 2 removes the numbers 2, 4, and 6 (exactly half of the numbers) and sieving by 3 removes just 3 ($\frac{1}{6}$ of the set, since the number 6 was already sieved in the first step).  Because of the repeating pattern of Erostosthenes sieve, the same fractions are found for the next finite set of integers (from 7 to 12) and so forth to infinity.  It is a more complicated process to find the exact number of prime numbers below a given primorial number, but the point here is that some proofs in sieve theory can be more easily done within a finite set.  Our goal here is not to delve into sieve theory, but to look at a simple series that was motivated by sieve methods, where primorial numbers arise naturally.

With this introduction, and a bit of thought, it is perhaps not surprising that the series given below converges.  The prime numbers become more sparse for larger integers, and so the fraction removed by each step of the sieve becomes smaller. But why should the series given here converge exactly to unity instead of some other real number? But before going further, we challenge the reader to write a proof of this.  It's not easy!

\section{Series Based on Number Sieves}
\subsection{Prime Series}
Consider the series, made from primes:
\begin{equation}
\frac{1}{2}+\frac{1}{6}+\frac{2}{30}+\frac{8}{210}+\frac{48}{2310}+\cdots 
+ \frac{1 \cdot 2 \cdot 4 \cdots (p_{k-1}-1)}{2 \cdot 3 \cdot 5 \cdots p_k} 
\label{eq:sum1}
\end{equation}
which arose from our study of prime number sieves. 
We will analyze it rigorously next. 

As described in the introduction, we start with the integers greater than 1 and sieve out any multiple of 2, leaving the fraction $ \frac{1}{2}$. By sieving out multiples of 3 we eliminate $\frac{1}{3}$ of all odd numbers, or $\frac{1}{6}$ of all integers. Next, sieving out all multiples of 5 leaves $ \frac{1}{5} - \frac{1}{2*5} - \frac{1}{3*5} + \frac{1}{2*3*5} = \frac{2}{30}$, which accounts for any double-counting from previous sieving. This gives us the terms of our series. More rigorously, we can define each term in the series $T_n$ as
\begin{equation} T_n = \frac{\phi (p_{n-1}\#)}{p_n\#} 
= \frac{1}{2},\,\frac{1}{6},\,\frac{2}{30},\,\frac{8}{210},\, \cdots 
\quad \textrm{for} \quad n = 1,2,3,\dots
\label{eq:term1}
\end{equation}
where $\phi(n)$ is Euler's totient function and $p_n\#$ is the n$^{th}$ primordial, the product of all primes up to $p_n$. We obtain this by noticing that the numerator is the number of numbers coprime to the primordial one order behind the one in our denominator. This can also be written as 

$$\frac{1 \cdot 2 \cdot 4 \cdots (p_{n-1}-1)}{2 \cdot 3 \cdot 5 \cdots p_n}$$

\begin{theorem}
Let $S_n$ be defined as the partial sum such that
$$ S_n  = \sum^{n}_{i=1} T_i, $$
then,
$$\lim_{n\to\infty} S_n = 1$$.
\label{thm:one}
\end{theorem}

Before proving that this series converges, we examine the series (\ref{eq:sum1}) 
term-wise and find the interesting relation
\begin{equation}
    T_n  = \frac{1-S_n}{p_n-1}.
\label{eq:Tn}
\end{equation}
For example, $T_2 = 1/6 = (1-2/3)/(3-1)$. We will soon prove Eq.~(\ref{eq:Tn}).\\
{\it Remark.} Note that $S_n$ can be written as
\begin{equation}
    S_n  = \sum^{n}_{k=1} \frac{\prod^k_{i=1}(p_i - 1)}
	{(p_k - 1)\prod^k_{i=1} (p_i)}
\label{eq:Sn}
\end{equation}
The term-wise relation is interesting because every term of the series 
includes the partial sum up to that term, implying a recursive nature. 

\begin{lemma}
Let $S_n$ be the series given by (1) and $p_i$ be the i$^{th}$ prime, then 
$$
    1-S_{n} = \prod^n_{i=1} (1-\frac{1}{p_{i}}).
$$
\end{lemma}
\begin{proof}
Inserting the definition of $T_n$ into (\ref{eq:Sn}) and rearranging, 
it can be shown true for the first few values of $n$.
For $n=1$, it is trivial.  For $n=2$
$$1-\frac{1}{p_1}-\frac{p_1-1}{p_1p_2} = (1-\frac{1}{p_1})(1-\frac{1}{p_2})$$
Next proceed by induction. Defining, 
\begin{equation}\label{eq:3}
1 - S_{n+1} = (1-S_{n}) - T_{n+1}
\label{lemma1}
\end{equation}
Substituting (\ref{eq:Tn}) and (\ref{eq:Sn}), we get
\begin{equation}
    1-S_{n+1} = \prod^n_{i=1}(1-\frac{1}{p_{i}}) - 
    \frac{\prod^{n+1}_{i=1}(p_{i}-1)}{(p_{n+1}-1)\prod^{n+1}_{i=1}(p_{i})} \ .
\end{equation}
Rewrite the right side so that all the indices of the products are the same, 
and cancel terms to get
\\
\begin{equation}
        1-S_{n+1} = \prod^{n}_{i=1}(1-\frac{1}{p_{i}}) - 
	\frac{\prod^{n}_{i=1}(p_{i}-1)}{(p_{n+1})\prod^{n}_{i=1}(p_{i})}
\end{equation}
\\
Factoring the right side and doing some algebra leads to 
\begin{equation}\label{eq:6}
 1-S_{n+1} = (1- \frac{1}{(p_{n+1})})\prod^{n}_{i=1}(1-\frac{1}{p_{i}})
\end{equation}
which is equivalent to
\begin{equation}
  1-S_{n+1} = \prod^{n+1}_{i=1}(1-\frac{1}{p_{i}})  
\end{equation}
\end{proof}

\noindent\begin{proof}[Theorem \ref{thm:one}]
Using Lemma \ref{lemma1}, 
\begin{equation}
    S_{n} = 1 - \prod^{n}_{i=1}(1-\frac{1}{p_{i}}) .
\end{equation}
This is a product that arises in Merten's Third Theorem \cite{Merten}, 
giving
\begin{equation}
    S_{n} = 1 - \frac{e^{\gamma}}{log(p_n)}
\end{equation}
where $\gamma$ is Euler's constant.
The second term vanishes in the limit $n \rightarrow \infty$, leaving
\begin{equation}
    S_{n} = 1
\end{equation}
\end{proof}

We know intuitively that this is the case, since by the Prime Number Theorem we know that the density of primes tends toward zero. In a more rigorous statement, the ratio all non-prime numbers to all integers tends asymptotically to 1. After the above conclusion, we next looked to expand our analysis to other number theoretic sieves.

\subsection{Square-Free Numbers}
Next, we examine the square free numbers. As with the primes, we begin with the integers greater than 1. Sieving out any number divisible by $2^2$, this will eliminate $\frac{1}{4}$ of the integers. Next we sieve by $3^2$, eliminating $\frac{1}{3^2} - \frac{1}{2^2*3^2}$ to account for double counting. We can already see that this eliminates a smaller fraction of the total number line than primes, and should lead to a more interesting result. Each term in this series of square-free (SF) numbers can then be written as 
$$T^{SF}_n = \frac{1^2 \cdot 2^2 \cdot 4^2 \cdots (p_{n-1}^2-1)}{2^2 \cdot 3^2 \cdot 5^2 \cdots p_n^2}$$

\begin{theorem}
Let $S^{SF}_n$ be defined as the partial sum such that
$$ S^{SF}_n  = \sum^{n}_{i=1} T^{SF}_i, $$
then,
$$\lim_{n\to\infty} S^{SF}_n = 1-\frac{6}{\pi^2}$$.
\label{thm:SquareFreeSum}
\end{theorem}
\begin{proof} Because this is of the same form as the first series, only with squares in the position of the primes, the same steps in the above proofs follow. Using a result similar to Lemma \ref{lemma1}, we can write the final series as 

\begin{equation}
    S^{SF}_{n} = 1 - \prod^{n}_{i=1}(1-\frac{1}{p_{i}^2}) .
\end{equation}

This product has a well-known value related to the Riemann zeta function, where 

$$\prod^{n}_{i=1}(1-\frac{1}{p_{i}^2}) = \frac{1}{\zeta(2)} = \frac{6}{\pi^2}$$

it follows that

\begin{equation}
    S^{SF}_{n} = 1 - \frac{6}{\pi^2} .
\end{equation}

\end{proof}
\subsection{Twin Primes}
We also decided to look at sieving into twin primes (TP). Twin primes are pairs 
of numbers such that for some p, both p and p+2 are prime. 
We now count the proportion of primes taken out at any stage 
of the sieve that are twin primes. We can do this by defining another set
for $n>2$,
\begin{equation}
    B_n = \{ x \leq p_n\# \, | 
    \, gcd(x,p_n\#)=p_n\ and\ gcd(x\pm 2,p_n\#)=1 \}
\end{equation}
This allows us to define another sequence as
\begin{equation} 
T^{TP}_n = \frac{|B_n|}{p_n\#} = 
\frac{1}{3},\,\frac{1}{15},\,\frac{3}{105},\,
\frac{15}{1155},\, \cdots \quad \textrm{for} \quad n = 2,3,4,\dots
\end{equation}
These are the fractions of pairs (x,x+2) {\it mod} $p_n$ which are removed 
at each step of the sieve (out of a total of $p_n\# /2$ possible odd pairs).
When taking the sum, the series becomes
\begin{equation}
    S^{TP}_n = \sum^n_{k=2} T^{TP}_k = \sum^n_{k=2} \frac{\prod^k_{i=2} (p_i-2)}
	{(p_k-2)\prod^k_{i=2}  (p_i)}
\label{eq:s2n}
\end{equation}
This series, following the same method as above, is shown to converge.

\begin{theorem}
Let 
$$S^{TP}_n = \frac{1}{3} + \frac{1}{15} + \frac{3}{105} + \cdots + 
  \frac{1\cdot 3 \cdot 5 \cdots (p_{n-1}-2)}{3\cdot 5\cdot 7 \cdots p_n}$$
then
$$    \lim_{n\to\infty} S^{TP}_n   = \frac{1}{2} $$
\label{thm:TwinPrimeSum}
\end{theorem}
\begin{proof}
We follow the same steps as for Theorem \ref{thm:one}, except now the numerator 
replaces $(p_i - 1)$ by $(p_i - 2)$, as in (\ref{eq:s2n}) above, 
and the relation between the sum and the product is

$$  \frac{1}{2} - S^{TP}_n = \frac{1}{2} \prod^{n}_{i=2} (1 - \frac{2}{p_i}) $$
and follow the same steps as for Lemma 1.
The product on the RHS is known to vanish \cite{Riesel} in the limit 
$n\to \infty$.
\end{proof}

This lead to seeking the general form of these series, which we examine next. 

\section{General Series}
\begin{theorem}
Let $F_i$ be an arbitrary sequence and let 
$$ S_n = \sum^{n}_{k=1}\frac{\prod^{k}_{i=1}(F_{i}-a)}
       {(F_{k}-a)\prod^{k}_{i=1}(F_{i})},
$$
then,
$$ \frac{1}{a}-S_{n} = \frac{1}{a} \cdot 
\frac{\prod^{n}_{i=1}(F_{i}-a)}{\prod^{n}_{i=1}(F_{i})}
$$
for some constant $a$.
\label{thm:general}
\end{theorem}

After deriving the results from twin primes, we noticed that Lemma \ref{lemma1} was 
a valid relation for not just the primes, but any for general series defined
with the n$^{th}$ term defined by
\begin{equation}
T_n = \frac{\prod^n_{i=1}(F_{i}-a)}{(F_{n}-a)\prod^n_{i=1}(F_{i})}
\label{eq:TnGen}
\end{equation}
where $a$ is some constant. 
We will prove Theorem \ref{thm:general} for this series in the same way we did 
previously, using proof by induction.

\noindent\begin{proof}[Theorem \ref{thm:general}]
We start with the first few terms, 
\begin{equation}
    [n=1]: \frac{1}{a} - \frac{1}{F_{1}} = \frac{F_{1}-a}{aF_{1}}
\end{equation}
\begin{equation}
    [n=2]: \frac{1}{a} - \frac{1}{F_{1}} - \frac{F_{1}-a}{F_{1}F_{2}}= \frac{(F_{1}-a)(F_{2}-a)}{aF_{1}F_{2}}
\end{equation}
We then assume Theorem \ref{thm:general} is true for n and examine the case for n+1.
\begin{equation}
    \frac{1}{a}-S_{n+1} = \frac{1}{a}-S_{n} - T_{n+1}.
\end{equation}
We then make substitutions noting that $T_{n}$ is the nth term in our series
\begin{equation}
   \frac{1}{a}-S_{n+1} = \frac{\prod^{n}_{i=1}(F_{i}-a)}{a\prod^{n}_{i=1}(F_{i})} - \frac{\prod^{n+1}_{i=1}(F_{i}-a)}{(F_{n+1}-a)\prod^{n+1}_{i=1}(F_{i})}.
\end{equation}
Then, we can expand the right side and collect terms to obtain
\begin{equation}
   \frac{1}{a}-S_{n+1} = \frac{\prod^{n}_{i=1}(F_{i}-a)}{a\prod^{n}_{i=1}(F_{i})} - \frac{\prod^{n}_{i=1}(F_{i}-a)}{(F_{n+1})\prod^{n}_{i=1}(F_{i})}.
\end{equation}
Factoring terms and rearranging gets us,
\begin{equation}
   \frac{1}{a}-S_{n+1} = \frac{\prod^{n}_{i=1}(F_{i}-a)}{a\prod^{n}_{i=1}(F_{i})}(1- \frac{a}{(F_{n+1})})
\end{equation}
Which, after some algebra, can be written as
\begin{equation}
   \frac{1}{a}-S_{n+1} = \frac{\prod^{n+1}_{i=1}(F_{i}-a)}{a\prod^{n+1}_{i=1}(F_{i})}
\end{equation}
\end{proof}

\section{Discussion and Applications}
The astute reader will have noticed that Theorem \ref{thm:general} is simply a 
consequence of expanding the product on the RHS using the binomial theorem. 
Rearranging Theorem \ref{thm:general} gives us a more useful form 
\begin{equation}
    \frac{1}{a}-\sum^{n}_{k=1}\frac{\prod^{k}_{i=1}(1-\frac{a}{F_{i}})}{(F_{i}-a)} = \frac{1}{a}\prod^n_{i=1}(1-\frac{a}{F_{i}}).
\label{eq:useful}
\end{equation}
We can use this to show the general form of the relation (\ref{eq:Tn}), 
by substituting in (\ref{eq:TnGen}) we obtain 
\begin{equation}
\frac{\frac{1}{a}-S_{n}}{F_n - a} = T_{n}.
\label{eq:recursive}
\end{equation}
This relation implies the same intrinsic recursive relation in the series 
such that each n$^{th}$ term contains the related n$^{th}$ partial sum. 
Furthermore, if we take $a=1$ then the right side of (\ref{eq:useful}) 
becomes a familiar form for many series, 
\begin{equation}
    \prod^{n}_{i=1}(1-\frac{1}{F_{i}}).
\label{eq:familiar}
\end{equation}
This form can be found in Mertens' theorems when $F_n$ is taken to be 
the sequence of primes. Other well known examples include the evaluation 
for $F_n = n^2$, then (\ref{eq:familiar}) is equal to $\frac{1}{2}$, 
and for $F_n = n$, then (\ref{eq:familiar}) is equal to 0 in 
the limits as n goes to infinity. 
So, by (\ref{eq:useful}), this implies their respective sums, $S_n$ are 
equal to $\frac{1}{2}$ and 1, respectively (in the same limit). 

This relation can be used to examine other interesting sequences such as 
the twin primes. 
We are interested in the infinite product that results from taking one 
minus the reciprocals of the twin primes, similar to the form found in 
Merten's Third Theorem, and finding whether or not it converges.
We begin by defining the twin prime pairs
$$(3,5),(5,7),(11,13),(17,19),(29,31),...$$
and arranging them into a sequence we define
$$p2=\{3,5,5,7,11,13,19,29,31,...\}$$
to be consistent with Brun's Theorem (see below). 
We can then describe our result.
\begin{theorem}
The product
\begin{equation}
  K = \prod^\infty_{i=1}(1-\frac{1}{p2_{i}}).
\end{equation}
converges to a constant, where K is bounded from above using Brun's constant.
\label{thm:K}
\end{theorem}

\noindent\begin{proof}
To evaluate this and prove its convergence we construct the left hand side 
of (\ref{eq:recursive}), taking $a=1$ in Theorem \ref{thm:general},
\begin{equation}
\label{eq:twinprod}
    S_n  = \sum^n_{k=1} \frac{\prod^k_{i=1}(p2_i-1)}
	{(p2_i-1)\prod^k_{i=1} (p2_i)}
\end{equation}
We will compare this series to the series found in Brun's Theorem \cite{Brun}. 
First, we will write the series from Brun's Theorem as 
\begin{equation}
    B_n  =  \sum^n_{k=1} \frac{1}{(p2_k)}
\end{equation}
Where $B$ goes to Brun's constant as $n \to \infty$, is $1.90216\dots$.
Note that Brun originally defined the sum over all twin pairs, 
$(1/3+1/5) + (1/5+1/7) + \dots$, giving a repeated value of $1/5$ in the sum.
The two series can then be compared term-wise. 
We begin by rewriting (\ref{eq:twinprod}) as 
\begin{equation}
    S_n  = \sum^n_{k=1} \frac{1}{p2_k} 
	\frac{\prod^{k-1}_{i=1}(p2_i-1)}{\prod^{k-1}_{i=1} (p2_i)}
\end{equation}
Using the fact that
\begin{equation}
   \prod^{k-1}_{i=1}(p2_i-1) < \prod^{k-1}_{i=1} (p2_i)
\end{equation}
for all k,  
this shows each term in $S_n$ will be less than each term in $B_n$. 
Because $B_n$ is known to converge, $S_n$ must likewise converge. 
Using Theorem \ref{thm:general} we can relate the series to the product, and so 
both series and product converge. 
\end{proof}

The constant $K$ in Theorem \ref{thm:K} can then be approximated heuristically. Using 
a computer program, we obtain the estimate
\begin{equation}
     K = \prod^{\infty}_{i=1}(1-\frac{1}{p2_{i}}) \approx 0.12933717\dots
\label{eq:K}
\end{equation}
The numerical methods used to evaluate this constant will be presented 
in a future paper \cite{WardHicks}.
 
This constant is of interest in the same way that Brun's constant is of 
interest, in that there is a comparison to the evaluation of the series 
of primes. The sum of the reciprocals of primes is divergent, yet the sum 
of the reciprocals of twin primes converges to Brun's constant \cite{Brun}. 
Similarly, we have shown here that in contrast to the product found in 
Merten's third theorem which converges to zero, the product (\ref{eq:K}) 
converges to a non-zero constant. 
This is another proof that the twin primes are more sparsely populated 
than primes, although it is still not known whether or not there are an 
infinite number of twin primes.

\end{document}